\documentclass[11pt]{article}
\usepackage{IEEEtrantools}
\usepackage{mathtools, nccmath}
\usepackage{latexsym}
\usepackage{amsfonts}
\usepackage{amssymb}
\usepackage{psfrag}
\usepackage{graphicx}
\usepackage{epsfig}
\usepackage{amsmath,amsfonts,amsthm}
\usepackage{mathrsfs}
\usepackage{enumerate}
\usepackage{dsfont}
\usepackage{hyperref}
\usepackage{pst-node}
\usepackage{cleveref}
\usepackage{enumitem}
\usepackage{tikz-cd}
\usepackage{multirow}
\usepackage{tikz}
\usetikzlibrary{arrows,decorations.markings}
\usepackage{bookmark}
\usepackage{hyperref}
\hypersetup{
     colorlinks   = true,
     citecolor    = blue
}

\newcommand\restr[2]{{
  \left.\kern-\nulldelimiterspace 
  #1 
  \littletaller 
  \right|_{#2} 
  }}

\newcommand{\littletaller}{\mathchoice{\vphantom{\big|}}{}{}{}}

\newcommand{\be}{\begin{equation}}
\newcommand{\ee}{\end{equation}}
\newcommand{\bea}{\begin{eqnarray}}
\newcommand{\eea}{\end{eqnarray}}
\newcommand{\bean}{\begin{eqnarray*}}
\newcommand{\eean}{\end{eqnarray*}}
\newcommand{\brray}{\begin{array}}
\newcommand{\erray}{\end{array}}
\newcommand{\biearray}{\begin{IEEEarray}{rCl}}
\newcommand{\eiearray}{\end{IEEEarray}}

\newtheorem{dfn}{Definition}[section]
\newtheorem{thm}[dfn]{Theorem}
\newtheorem{lmma}[dfn]{Lemma}
\newtheorem{ppsn}[dfn]{Proposition}
\newtheorem{crlre}[dfn]{Corollary}
\newtheorem{xmpl}[dfn]{Example}
\newtheorem{rmrk}[dfn]{Remark}

\newcommand{\bdfn}{\begin{dfn}\rm}
\newcommand{\bthm}{\begin{thm}}
\newcommand{\blmma}{\begin{lmma}}
\newcommand{\bppsn}{\begin{ppsn}}
\newcommand{\bcrlre}{\begin{crlre}}
\newcommand{\bxmpl}{\begin{xmpl}}
\newcommand{\brmrk}{\begin{rmrk}\rm}

\newcommand{\edfn}{\end{dfn}}
\newcommand{\ethm}{\end{thm}}
\newcommand{\elmma}{\end{lmma}}
\newcommand{\eppsn}{\end{ppsn}}
\newcommand{\ecrlre}{\end{crlre}}
\newcommand{\exmpl}{\end{xmpl}}
\newcommand{\ermrk}{\end{rmrk}}

\newcommand{\bbc}{\mathbb{C}}

\newcommand{\tr}{\mathrm{tr}}

\newcommand{\id}{\mbox{id}}

\def \qed { \mbox{}\hfill
$\Box$\vspace{1ex}}

\title{Sections and Chapters}

\begin{document}
	
	
	\author{\sc {Keshab Chandra Bakshi,}\\[2pt]{\sc Debashish Goswami,}\\[2pt]{\sc Biplab Pal}}

	\title{Weak quantum hypergroups from finite index $C^*$-inclusions}

	\maketitle
	
	
\begin{abstract}
We study a finite index inclusion $B \subset A$ of simple unital $C^{*}$-algebras and construct a canonical completely positive coproduct on the second relative commutant $B' \cap A_{1}$, thereby endowing it with a natural coalgebra structure. Motivated by this construction, we introduce the notion of a \emph{weak quantum hypergroup}, a generalization of the quantum hypergroups of Chapovsky and Vainerman [Compact quantum hypergroups, J. Operator Theory 41 (1999), no.~2, 261–289]. We show that every finite index inclusion gives rise to such a weak quantum hypergroup, and that the corresponding weak quantum hypergroup possesses a Haar integral. In the irreducible case, this structure satisfies the axioms of a quantum hypergroup in the sense of Chapovsky and Vainerman, while in the depth $2$ setting our framework yields the associated weak Hopf algebra constructed by Nikshych and Vainerman. These results provide a unified and intrinsically $C^{*}$-algebraic framework for generalized quantum symmetries associated with finite index inclusions.
\end{abstract}

	\bigskip
	
	{\bf AMS Subject Classification No.:} {\large 46}L{\large 37}\,, {\large 47}L{\large 40}\,, {\large 46}L{\large 05}\,, {\large 43}A{\large 30}\,.
	
	\bigskip
	\hypersetup{linkcolor=blue}
	
\section{Introduction}
Jones introduced the notion of the index \([M : N]\) for a unital inclusion of type \(II_1\) factors \(N \subset M\) in \cite{J}. This breakthrough initiated an active field of research in operator algebras, with significant applications to various areas of mathematics and mathematical physics. Later, Kosaki extended the definition of index to inclusions of type $III$ factors \cite{K}. Inclusion of simple $C^*$-algebras encompasses both type $II_1$ and type $III$ subfactors. Watatani subsequently developed a very general algebraic approach to index theory in this setting \cite{Watataniindex}, defining an index for inclusions of \(C^*\)-algebras possessing a conditional expectation that admits a quasi-basis, a generalization of the Pimsner–Popa basis \cite{PP}. Inclusions of $C^*$-algebras play a central role in operator algebra theory, serving as a unifying framework for index theory, quantum symmetries, and tensor categorical structures. Finite index inclusions, in particular, encode rich internal symmetry and give rise to tensor categories of bimodules, which have become a fundamental tool in subfactor theory, fusion categories, and low-dimensional topology; see, for example, \cite{BCG,BVG,HP,I,M,PN,R} for recent developments. Understanding how such symmetries can be canonically extracted from an inclusion of simple $C^*$-algebras, without imposing restrictive assumptions such as irreducibility or depth conditions, is a key motivation for the present work. 

For a subfactor \(N \subset M\) with finite index \([M:N] < \infty\), Jones observed that one can construct another type \(II_1\) factor \(M_1\), referred to as the basic construction, satisfying \([M_1:M] = [M:N]\). This procedure can be iterated to form the fundamental Jones tower:
\[
N \subset M \subset M_1 \subset M_2 \subset \cdots \subset M_k \subset \cdots.
\]
The existence of this tower was a key ingredient in the proof of the celebrated index rigidity theorem in \cite{J}. 
Since the inception of subfactor theory, the higher relative commutants \(N' \cap M_k\) and \(M' \cap M_k\) play a crucial role in understanding the subfactors. Ocneanu introduced the notion of a Fourier transform $\mathcal{F}$ from $N' \cap M_1$ onto $M' \cap M_2$, generalizing the classical Fourier transform for finite abelian groups (see also \cite{B,BJ}). Through this transformation, one obtains a novel multiplication on \(N' \cap M_1\), which extends the classical convolution product. Motivated by these developments, the first-named author together with Gupta introduced in \cite{BakshiVedlattice} a convolution product on the higher relative commutants $B' \cap A_k$ and $A' \cap A_k$, arising from a unital inclusion of simple unital $C^*$-algebras $B\subset A$ with a conditional expectation of index-finite type, where
\[
B \subset A \subset A_1 \subset A_2 \subset \cdots \subset A_k \subset \cdots
\]
denotes Watatani’s tower of $C^*$-basic constructions with respect to the unique minimal conditional expectation. In \cite[Lemma~3.20]{BakshiVedlattice}, they further established that this convolution product is associative. Consequently, the second relative commutant acquires a natural associative algebra structure. In the present work, motivated by  \cite{HLPW}, we construct a coalgebra structure on the second relative commutant $B' \cap A_1$ using this convolution product, and we prove that the associated coproduct map is completely positive.
\medskip

\noindent\textbf{Theorem A} (See \Cref{impformainthm6})  
Let \( B \subset A \) be an inclusion of simple unital \( C^* \)-algebras equipped with a conditional expectation of index-finite type. Equip \( B' \cap A_1 \) with the inner product $\langle x, y \rangle := \operatorname{tr}(x^* y).$ Define a linear map
\[
\Delta \colon B' \cap A_1 \longrightarrow (B' \cap A_1) \otimes (B' \cap A_1)
\]
by
\[
\langle \Delta(x),\, y \otimes z \rangle
= \tau^{\frac{1}{2}} \, \langle x,\, y \star z \rangle,
\qquad x,y,z \in B' \cap A_1,
\]
where \( \star \) denotes the convolution product and $\tau := [A:B]_0^{-1}.$ Then the map \( \Delta \) is completely positive.

\medskip

 Hopf algebras and their generalizations play a fundamental role in the theory of subfactors and fusion categories, as they provide algebraic models for quantum symmetries. In the depth 2 setting, finite index subfactors are completely described by finite-dimensional weak Hopf algebras (also called quantum groupoids), as shown by Nikshych and Vainerman in \cite{NV}. In this case, the subfactor can be realized as a crossed product by a weak Hopf algebra, and the associated fusion category of bimodules is equivalent to the representation category of this quantum symmetry object.
 For a depth 2 inclusion of simple unital $C^*$-algebras with a finite index conditional expectation, a similar structural result has also been recently established (see \cite{I,preprint}). Beyond depth 2, fusion categories arising from subfactors no longer admit a single Hopf algebra symmetry. Nevertheless, remnants of Hopf algebraic structure persist in the form of bimodule categories, paragroups, and planar algebras. 
A quantum hypergroup (in the sense of Chapovsky and Vainerman in \cite{CV}) is a unital $C^*$-algebra equipped with a completely positive, coassociative coproduct. The most important examples of such structures arise from double cosets of compact matrix pseudogroups in the sense of Woronowicz (\cite{W}). Chapovsky and Vainerman further showed in \cite{CV} that quantum hypergroups possess rich additional structure. As a first step toward developing a general framework for arbitrary inclusions of simple unital $C^*$-algebras, we introduce the notion of a weak quantum hypergroup by relaxing some of the axioms of a quantum hypergroup. The present work fits naturally into this landscape by introducing weak quantum hypergroups as a further generalization of Hopf and weak Hopf algebra symmetries. Weak quantum hypergroups retain the essential analytic features of quantum symmetries- such as coproducts, Haar integrals, and actions- while replacing multiplicativity by complete positivity. This allows one to associate a canonical symmetry object to arbitrary finite index inclusions of  $C^*$-algebras, extending the reach of Hopf-algebraic methods to settings where genuine Hopf or weak Hopf algebra symmetries are unavailable. More precisely, we prove the following:

\medskip
\noindent\textbf{Theorem B:} (See Theorems \ref{imp for Theorem B}, \ref{forirr}, \ref{normalizedmeasures}, and \ref{Haarintegralexistence}) Let \(B \subset A\) be an inclusion of simple unital \(C^*\)-algebras with a conditional expectation of index-finite type. In this paper, we establish the following results:
\begin{itemize}
  \item We show that the second relative commutant $B' \cap A_1$, associated with the basic construction, carries a canonical \emph{weak quantum hypergroup structure} naturally arising from the inclusion $B \subset A$.
  
  \item We prove that, when the inclusion $B \subset A$ is irreducible, the weak quantum hypergroup structure on $B' \cap A_1$ upgrades to that of a (genuine) quantum hypergroup.
  
  \item We further show that the weak quantum hypergroup $B' \cap A_1$ admits left- and right-invariant measures, and that it also admits a Haar integral.
\end{itemize}

The main contribution of this paper is the introduction of a new symmetry object---the {weak quantum hypergroup}---which arises canonically from a finite index inclusion of simple unital $C^*$-algebras. This concept provides a unified operator-algebraic framework that extends and connects previously known quantum symmetry structures, while remaining valid in complete generality. Even though a weak quantum hypergroup does not encode all higher relative commutants $B' \cap A_n$ in the Jones tower, it nevertheless captures, in a very explicit manner,  the canonical first layer of the quantum symmetry inherent in the inclusion. Moreover, our construction unifies all known special cases: in the irreducible setting the weak quantum hypergroup satisfies the axioms of a quantum hypergroup in the sense of Chapovsky and Vainerman, while in the depth 2 case it collapses to the weak Hopf algebra of Nikshych and Vainerman. 

\section{Preliminaries}\label{preliminaries}
This section collects a number of basic concepts that are employed in the sequel. Our discussion is necessarily brief, and we refer the reader to the literature for more detailed accounts.

\subsection{Inclusions of simple unital $C^*$-algebras}
We begin by recalling Watatani’s $C^*$-index theory for unital inclusions of simple $C^*$-algebras, following \cite{Watataniindex}. We then briefly review the Fourier theory developed in \cite{BakshiVedlattice,BGPS,BGS} for the corresponding relative commutants. Let $B \subset A$ be an inclusion of simple unital $C^*$-algebras equipped with a conditional expectation of finite Watatani index. Consider the associated tower of $C^*$-basic constructions:
\begin{center}
    $B \subset A \subset A_1 \subset A_2 \subset \cdots \subset A_n \subset \cdots$
\end{center}
with unique (dual) minimal conditional expectations $E_n : A_n \to A_{n-1}$ for $n \geq 0$, where we adopt the convention that $A_{-1}:=B$ and $A_0:=A$. For each $n \geq 1$, let $e_n$ denote the Jones projection in $A_n$. We denote $\tau := [A:B]^{-1}_0$.
\begin{lmma}[\cite{BakshiVedlattice}]\label{pushdown} For every \(x_1 \in A_1\) there exists a unique \(x_0 \in A\) such that \(x_1 e_1 = x_0 e_1\). In fact, $x_0 = \tau^{-1} E_1(x_1 e_1).$
\end{lmma}
\begin{lmma}[\cite{KajiwaraWatatani}]\label{kajiwara watatani}
Let \(B \subset A\) be a unital inclusion of simple \(C^*\)-algebras, and let \(E_0 : A \to B\) denote the minimal conditional expectation. Then, for any \(x \in B' \cap A\) and \(a \in A\), we have $E_0(xa) = E_0(ax).$
\end{lmma}
 
 The relative commutants $B' \cap A_n := \{ x \in A_n : xb = bx \text{ for all } b \in B \}$ are finite-dimensional $C^*$-algebras (see \cite{Watataniindex}). The inclusion \( B \subset A \) is said to be of finite depth (cf.\ \cite{JOPT}) if there exists an integer \( n \ge 1 \) such that $(B' \cap A_{n-1})\, e_n\, (B' \cap A_{n-1}) = B' \cap A_n.$ The smallest such integer \( n \) is called the depth of the inclusion.

\medskip

  Using the minimal conditional expectations, one obtains ``Markov-type traces'' on these algebras \cite[Proposition~2.21]{BakshiVedlattice}. Specifically, for $n \geq 0$, the map
\[
\mathrm{tr}_n := (E_0 \circ E_1 \circ \cdots \circ E_n)\big|_{B' \cap A_n} : B' \cap A_n \to \mathbb{C}
\]
defines a faithful tracial state on $B' \cap A_n$.
\begin{ppsn}[\cite{BakshiVedlattice}]\label{m1} 
For all \(n \ge 1\), $\tr_n(x e_n) = \tau \, \tr_{n-1}(x) \text{ for every } x \in B' \cap A_{n-1},$ and $\tr_n|_{B' \cap A_{n-1}} = \tr_{n-1}.$ \end{ppsn}
 When there is no risk of confusion, we shall drop the subscript and write simply $\tr$ for notational convenience. If $\{\lambda_i : 1 \leq i \leq n\} \subset A$ is a quasi-basis for the minimal conditional expectation $E_0$. Then the $\tr$-preserving conditional expectation from $B' \cap A_n$ onto $A' \cap A_n$ is given by the following:
\[
E_{A' \cap A_n}^{B' \cap A_n}(x)
    = \tau \sum_i \lambda_i x \lambda_i^*
    \qquad \text{for all } x \in B' \cap A_n.
\]

The idea of a Fourier transform for paragroups associated with finite depth subfactors was first proposed by Ocneanu \cite{O}, and has since played a central role in subfactor theory.
A formula for the Fourier transform on the higher relative commutants of an arbitrary extremal subfactor was obtained by Bisch in \cite{B}. Following \cite{BakshiVedlattice} (see also \cite{BGPS,BGS}), we briefly review the Fourier transform and the associated convolution structure on the relative commutants for $C^*$-inclusions, where these tools have found several applications (see, for instance, \cite{BakshiVedlattice,BGP,BGPS,BGS}). For each \(n \ge 0\), the Fourier transform $\mathcal{F}_n : B' \cap A_n \to A' \cap A_{n+1} $ is defined by
\[
\mathcal{F}_n(x)
    = \tau^{-\frac{n+2}{2}}
      E^{B' \cap A_{n+1}}_{A' \cap A_{n+1}}(x v_{n+1}),
    \qquad x \in B' \cap A_n,
\]
where \(v_{n+1} = e_{n+1} \cdots e_1\). The inverse Fourier transform
\(\mathcal{F}^{-1}_n : A' \cap A_{n+1} \to B' \cap A_n\)
is given by
\[
\mathcal{F}^{-1}_n(w)
    = \tau^{-\frac{n+2}{2}} \, E_{n+1}(w v_{n+1}^*),
    \qquad w \in A' \cap A_{n+1}.
\]
This terminology is justified because, for every \(n \ge 0\), $\mathcal{F}_n \circ \mathcal{F}^{-1}_n = \mathrm{id}_{A' \cap A_{n+1}}
\quad \text{and} \quad
\mathcal{F}^{-1}_n \circ \mathcal{F}_n = \mathrm{id}_{B' \cap A_n},$
as proved in \cite[Theorem~3.5]{BakshiVedlattice}. For \(x, y \in B' \cap A_n\), the convolution product is defined by
\begin{equation}
x \star y
    := \mathcal{F}_n^{-1}\big( \mathcal{F}_n(y)\, \mathcal{F}_n(x) \big),
\end{equation}
which is associative by \cite[Lemma~3.20]{BakshiVedlattice}. Furthermore, if
\(x, y \in B' \cap A_1\),
\begin{equation}\label{3Proposition 3.8}
(x \star y)^* = x^* \star y^*,
\end{equation}
as stated in \cite[Proposition~3.8]{BGS}. However, this $*$-compatibility does not generally extend to higher relative commutants \cite[Remark~4.3]{BGPS}. For each \(n \ge 0\), the rotation operator
\(\rho_n^+ : B' \cap A_n \to B' \cap A_n\)
is defined by
\[
\rho_n^+(x)
    = \big(\,\mathcal{F}_n^{-1}(\mathcal{F}_n(x)^*)\,\big)^*,
    \qquad x \in B' \cap A_n.
\] 
Reflection operators, introduced in \cite{BGPS}, provide powerful tools for establishing Fourier-theoretic inequalities on the higher relative commutants and for connecting the Connes-Størmer entropy of the canonical shift with the minimal Watatani index. For every \(n \ge 0\), $r_{2n+1}^+ := (\rho_{2n+1}^+)^{\,n+1},$ defines a map \(r_{2n+1}^+ : B' \cap A_{2n+1} \to B' \cap A_{2n+1}\). For $n=1$, these operators admit the following concrete expression that will be used later. For all \(x \in B' \cap A_1\),
\begin{equation}\label{reflection operator formula}
r_1^+(x)
    = \tau^{-1} \sum_i E_1(e_1 \lambda_i x)\, e_1 \lambda_i^*.
\end{equation} The reflection operators enjoy the following structural properties:

\begin{thm}[\cite{BGPS}]\label{aboutreflectionoperator}
For every \(n \ge 0\), the reflection operator \(r_{2n+1}^+\) is a unital, involutive, \(*\)-preserving anti-homomorphism, and moreover it preserves the trace.
\end{thm}
In the setting of simple unital \(C^*\)-algebras, one also defines the shift operators, as introduced in \cite{BGPS}. For each \(n \ge 0\), $S_n^+ := r_{2n+3}^+ \circ r_{2n+1}^+.$
Each map $S_n^+ : B' \cap A_{2n+1} \rightarrow A^{\prime}_1 \cap A_{2n+3}$ is a unital, trace-preserving \(*\)-isomorphism. By \cite[Corollary 3.26]{BGPS}, the reflection operator $S^+_0$ admit the following explicit form:
\begin{equation}\label{shift formula}
S^+_0(x)=\tau^{-2}\sum_i\lambda_{i}e_1e_2ye_3e_2e_1\lambda^*_{i}, \quad \text{ for } x\in B'\cap A_1.
\end{equation}
\subsection{Quantum hypergroups}
We now recall some preliminaries on quantum hypergroups, as introduced by Chapovsky and Vainerman \cite{CV}. We also refer the reader to \cite{DV, DV2, LV} for further developments and generalizations of quantum hypergroup concepts.
\begin{dfn}[\cite{CV}]\label{def:hypergroup}
A quadruple \( (A, \Delta, \varepsilon, \star) \) is said to define a quantum hypergroup structure on the \( C^* \)-algebra
\( (A, \cdot, 1, *) \) if the following conditions are satisfied:
\begin{itemize}
\item
The system \( (A, \Delta, \varepsilon, \star) \) forms a \( * \)-coalgebra with counit \( \varepsilon \); that is, \( \Delta \colon A \to A \otimes A \)
and \( \varepsilon \colon A \to \mathbb{C} \) are linear maps, and
\( \star \colon A \to A \) is an antilinear map such that
\begin{enumerate}
\item[$(i)$]
$(\Delta \otimes \mathrm{id})\circ \Delta = (\mathrm{id} \otimes \Delta)\circ \Delta,$
\item[$(ii)$]
$(\varepsilon \otimes \mathrm{id})\circ \Delta = (\mathrm{id} \otimes \varepsilon)\circ \Delta = \mathrm{id},$
\item[$(iii)$]
$\Delta \circ \star = \Pi \circ (\star \otimes \star)\circ \Delta,$
\item[$(iv)$]
$\star \circ \star = \mathrm{id},$
\end{enumerate}
where \( \Pi \colon A \otimes A \to A \otimes A \) denotes the flip map.

\item
The comultiplication \( \Delta \colon A \to A \otimes A \) is positive.

\item
The following compatibility conditions are satisfied for all \( a,b \in A \):
\begin{enumerate}
\item[$(i)$]
$(a \cdot b)^\star = a^\star \cdot b^\star,$ 
\item[$(ii)$]
$\Delta \circ * = (* \otimes *)\circ \Delta,$
\item[$(iii)$]
$\varepsilon(a \cdot b) = \varepsilon(a)\varepsilon(b),$
 \item[$(iv)$]
 $\Delta(1) = 1 \otimes 1,$
\item[$(v)$]
$\star \circ * = * \circ \star.$
\end{enumerate}
\end{itemize}
\end{dfn}
\begin{dfn}[\cite{CV}]\label{def:Haar}
Let \( (A, \Delta, \varepsilon, \star) \) be a quantum hypergroup structure on a
\( C^* \)-algebra \( A \).
A state \( \phi\) is called a Haar measure if $\phi \otimes \id)\circ \Delta(a)
= (\id \otimes \phi)\circ \Delta(a)
= \phi(a),$ for all $a \in A.$
\end{dfn}
If \( A \) is a quantum hypergroup whose linear span of positive definite elements is
dense in \( A \), then \( A \) admits a unique Haar measure \( \phi \)
(see \cite[Theorem~2.3]{CV}). It is worth noting that in \cite{CV} the authors
further introduce and study the notion of compact quantum hypergroups, the most important examples of such structures arise from
double cosets of compact matrix pseudogroups. Several notions of hypergroups associated with subfactors have also been studied in the literature; see, for example, \cite{BVG,S}. Bischoff et al.~\cite{BVG} compare their construction with other hypergroup frameworks appearing in the literature, including those developed in \cite{BH,CV,SW}.

\section{Convolution product and coalgebra structure}	

In this section, we consider the inclusion \(B \subset A\) of simple unital \(C^*\)-algebras with a conditional expectation of index-finite type and define a coalgebra structure on the second relative commutant $B' \cap A_1$ using the convolution product. We then show that the associated coproduct map is completely positive. An application of this result will be presented in the next section. Finally, we prove that this coalgebra structure coincides, up to isomorphism, with the one introduced by Nikshych and Vainerman (\cite{NV,preprint}) in the case of depth 2 subfactor. This correspondence reinforces the naturality of our construction and highlights its relevance within the broader theory of generalized quantum symmetries arising from subfactor inclusions.

\medskip
\noindent\textbf{Notation:} Throughout this paper, we work with an inclusion \(B \subset A\) of simple unital \(C^*\)-algebras with a conditional expectation of index-finite type. Let \(\{\lambda_i : 1 \le i \le n\} \subset A\) be a quasi-basis for the minimal conditional expectation \(E_0\).

\begin{dfn}\label{weak hypergroup structure} On the second relative commutant $B'\cap A_1$, we define the maps $\Delta$ and  $\varepsilon$ as follows:
$$\Delta:B'\cap A_1 \rightarrow (B'\cap A_1) \otimes (B'\cap A_1):\quad\big<\Delta(x),y\otimes z\big> = \tau^{\frac{1}{2}}\big<x,y \star z\big>, $$
$$\varepsilon: B'\cap A_1 \rightarrow \bbc: \quad \varepsilon(x)=\tau^{-1}\big<e_1,x\big>,$$
for all  $x,y,z \in B'\cap A_1$. Here the inner product on $B'\cap A_1$ is given by $\big<x,y\big>= \tr(x^*y)$.

\end{dfn}
\begin{ppsn}\label{impformainthm1}
$(B'\cap A_1,\Delta, \varepsilon)$ becomes a coalgebra.
\end{ppsn}
\begin{prf}
For any $x,y_1,y_2,y_3 \in B'\cap A_1$, we have the following:
\begin{eqnarray}\nonumber
\big<(\Delta \otimes \mathrm{id})\circ \Delta(x), y_1\otimes y_2 \otimes y_3\big>
&=&\big<\Delta(x_{(1)}) ,y_1\otimes y_2\big>\,\big<x_{(2)},y_3\big>\\\nonumber
&=&\tau^{\frac{1}{2}}\big<x_{(1)} ,y_1 \star y_2\big>\,\big<x_{(2)},y_3\big>\\\nonumber
&=&\tau \big<x,(y_1 \star y_2)\star y_3\big>.
\end{eqnarray}
Similarly, it is easy to verify that, $\big<(\mathrm{id}\otimes \Delta)\circ \Delta(x), y_1\otimes y_2 \otimes y_3\big>=\tau \big<x,y_1\star( y_2\star y_3)\big>.$ By \cite{BakshiVedlattice}[Lemma 3.20], we know that the convolution product is associative. Hence we get $(\Delta \otimes \mathrm{id})\circ \Delta=(\mathrm{id} \otimes \Delta)\circ \Delta.$ Also we have,
\begin{eqnarray}\nonumber
\big<y_1,(\varepsilon\otimes \mathrm{id})\circ \Delta(x)\big>
&=&\varepsilon(x_{(1)})\big< y_1, x_{(2)}\big>\\\nonumber
&=&\tau^{-\frac{1}{2}}\big<e_1 \star y_1, x\big>\\\nonumber
&=&\big<y_1, x\big>.
\end{eqnarray}
Moreover, one readily verifies that $\big<y_1,(\mathrm{id}\otimes \varepsilon)\circ \Delta(x)\big>=\big<y_1, x\big>.$ So we have $(\varepsilon\otimes \mathrm{id})\circ \Delta=(\mathrm{id} \otimes \varepsilon)\circ \Delta=\mathrm{id}.$ Hence $(B'\cap A_1,\Delta, \varepsilon)$ becomes a coalgebra.
\qed
\end{prf}

As a direct consequence of \cite{T}, we obtain the following:
\begin{thm}\label{Takeuchi}
Let \(B \subset A\) be an inclusion of simple unital \(C^*\)-algebras with a conditional expectation of finite Watatani index. Then a Hopf algebra can be constructed from the coalgebra \((B' \cap A_1, \Delta, \varepsilon)\).
\end{thm}

We now aim to show that $\Delta$ is a completely positive map. In the framework of a finite index subfactor, this has been established in \cite{HLPW}, where the proof is given via pictorial arguments. In contrast, in the $C^*$-algebraic framework, the main obstacle is that a comparable pictorial description (as in planar algebras) for the Fourier transform and convolution is still not available. To overcome this difficulty, we instead exploit the relationship between quasi-basis and minimal conditional expectation. We begin by establishing the following preliminary results:
\begin{ppsn}\label{isomorphism}
The \(C^*\)-algebras \((B' \cap A_1) \otimes (B' \cap A_1)\) and \((B' \cap A_1)(A_1' \cap A_3)\) are isomorphic.
\end{ppsn}
\begin{prf}
Define \(T : (B' \cap A_1) \otimes (B' \cap A_1) \to (B' \cap A_1)(A_1' \cap A_3)\) by
\[
T(x \otimes y) = x\, S_0^+(y),
\qquad x,y \in B' \cap A_1.
\]
 where \(S_0^+\) is the shift operator on $B'\cap A_1$ as defined in \cite{BGPS}. Then clearly $T$ is an $*$-algebra homomorphism. We now observe that
\begin{eqnarray}\nonumber
\|\sum_{i} x_i \otimes y_i\|^2&=&\big<\sum_{i} x_i \otimes y_i, \sum_{j} x_j \otimes y_j\big>
=\sum_{i,j}\big< x_i \otimes y_i,x_j \otimes y_j\big>=\sum_{i,j}\big< x_i ,x_j\big>\big< y_i ,y_j\big>.
\end{eqnarray}
Moreover, we have
\begin{eqnarray}\nonumber
\|\sum_{i}x_i\,S^+_0(y_i)\|^2&=& \big<\sum_{i} x_i \,S^+_0( y_i), \sum_{j} x_j\,S^+_0 (y_j)\big>\\\nonumber
&=& \sum_{i,j}\big< x_i \,S^+_0( y_i),  x_j\,S^+_0 (y_j)\big>\\\nonumber
&=& \sum_{i,j} \tr( S^+_0( y^*_i)\,x^*_i x_j\,S^+_0 (y_j)\big>\\\nonumber
&=& \sum_{i,j} \tr( S^+_0 (y_j)S^+_0( y^*_i)\,x^*_i x_j).
\end{eqnarray}
The following observation holds.
 $$\tr( S^+_0 (y_j y^*_i)\,x^*_i x_j)=E_0\circ E_1\circ E_2 \circ E_3( S^+_0 (y_j y^*_i)\,x^*_i x_j)=E_0\circ E_1( E_2 \circ E_3( S^+_0 (y_j y^*_i))\,x^*_i x_j).$$
Using \cite[Theorem 3.27]{BGPS}, we know that $S^+_0 (y_j y^*_i) \in A^{\prime}_1\cap A_3$. Hence $E_2 \circ E_3( S^+_0 (y_j y^*_i)) \in \bbc$. Moreover, since $S^+_0$ is trace preserving, we obtain the following:
 \begin{eqnarray}\nonumber
 \|\sum_{i}x_i\,S^+_0(y_i)\|^2
&=& \sum_{i,j} E_0\circ E_1(x^*_i x_j) E_2 \circ E_3( S^+_0 (y_j y^*_i))\\\nonumber
&=& \sum_{i,j} \tr(x^*_i x_j) \tr( S^+_0 (y_j y^*_i))\\\nonumber
&=& \sum_{i,j} \tr(x^*_i x_j) \tr( y_j y^*_i)\\\nonumber
&=&\sum_{i,j}\big< x_i ,x_j\big>\big< y_i ,y_j\big>.
\end{eqnarray}
Hence, \(T\) is an isometry and therefore injective. Now let \(xw \in (B' \cap A_1)(A_1' \cap A_3)\). Since \(S_0^+\) is surjective, we may write \(xw = x\, S_0^+(y)\) for some \(y \in B' \cap A_1\). But then $T(x \otimes y) = x\, S_0^+(y) = xw,$ showing that \(T\) is surjective. This completes the proof.
\qedhere

\end{prf}
\begin{ppsn}\label{main theorem}
For \(z \in B' \cap A_1\), we have
\[
\Delta(z)
=
\tau^{-2}\, T^{-1} \circ E(e_2 z e_3 e_2),
\]
where \(E\) is the trace preserving conditional expectation from
\(B' \cap A_3\) onto \((B' \cap A_1)(A_1' \cap A_3)\).
\end{ppsn}

\begin{prf}
So it is enough to prove that for any $x,y \in B'\cap A_1$,
$$\tau^{-2}\,\big<T^{-1}\circ E(e_2ze_3e_2), x\otimes y \big>=\tau^{\frac{1}{2}}\big<z,    x\star y\big>.$$
A direct computation shows that,
$$\big<\sum_{i}x_i \otimes y_i, \sum_{j} x^{\prime}_j \otimes y^{\prime}_j\big>= \big<\sum_{i}x_i S^+_{0}(y_i) , \sum_{j} x^{\prime}_j S^+_{0}(y^{\prime}_j)\big> $$
Consequently, we obtain the following,
\begin{eqnarray}\nonumber
\tau^{-2}\,\big<T^{-1}\circ E(e_2ze_3e_2), x\otimes y \big>&=&\tau^{-2}\,\big< E(e_2ze_3e_2), x\,S^+_{0}(y) \big>
\end{eqnarray}
Since $x\,S^+_{0}(y)\in{(B'\cap A_1)(A^{\prime}_1\cap A_3)}$ and $E$ is the trace preserving conditional expectation, it follows that:
\begin{eqnarray}\nonumber
\tau^{-2}\,\big<T^{-1}\circ E(e_2ze_3e_2), x\otimes y \big>
&=&\tau^{-2}\,\tr(e_2e_3z^*e_2x\,S^+_{0}(y))\\\nonumber
&=&\tau^{-4}\,\sum_{i}\tr(e_2e_3z^*e_2x\lambda_{i}e_1e_2ye_3e_2e_1\lambda^*_{i})\quad (\text{by \Cref{shift formula})} \\\nonumber
&=&\tau^{-4}\,\sum_{i}\tr\big(e_2\,E_2(z^*e_2x\lambda_{i}e_1e_2y)e_3e_2e_1\lambda^*_{i}\big).
\end{eqnarray}
Noting that $E_2(z^*e_2x\lambda_{i}e_1e_2y)=E_2\big(z^*E_1(x\lambda_{i}e_1)e_2y\big)=\tau z^*E_1(x\lambda_{i}e_1)y,$ it follows that
\begin{eqnarray}\nonumber
\tau^{-2}\big<T^{-1}\circ E(e_2ze_3e_2), x\otimes y \big>
&=&\tau^{-3}\sum_{i}E_0\circ E_1\circ E_2 \circ E_3\big(e_2z^*E_1(x\lambda_{i}e_1)ye_3e_2e_1\lambda^*_{i}\big)\\\nonumber
&=&\tau^{-2}\sum_{i}E_0\circ E_1\circ E_2 \big(e_2z^*E_1(x\lambda_{i}e_1)ye_2e_1\lambda^*_{i}\big)\\\nonumber
&=&\tau^{-2}\sum_{i}E_0\circ E_1\circ E_2 \big(E_1(z^*E_1(x\lambda_{i}e_1)y)e_2e_1\lambda^*_{i}\big)\\\nonumber
&=&\tau^{-1}\sum_{i}E_0\circ E_1 \big(E_1(z^*E_1(x\lambda_{i}e_1)y)e_1\lambda^*_{i}\big)\\\label{LHS}
&=&\sum_{i}E_0 \circ E_1\big(z^*E_1(x\lambda_{i}e_1)y\lambda^*_{i}\big).
\end{eqnarray}
We now compute the following,
\begin{eqnarray}\nonumber
\big<z, x\star y\big>&=&\tr(z^*(x\star y))\\\nonumber
&=& \tr(z^*(x^* \star y^*)^*)  \quad \text{( by \Cref{3Proposition 3.8})}\\\nonumber
&=& \big(E_0\circ E_1\big((x^* \star y^*)z\big)\big)^*.
\end{eqnarray}
From the definition of the convolution product, it follows that
\begin{eqnarray}\nonumber
E_0\circ E_1((x^* * y^*)z)&=& E_0\circ E_1\big(\mathcal{F}^{-1}_1(\mathcal{F}_1(y^*)\mathcal{F}_1(x^*))z\big)\\\nonumber
&=& \tau^{-\frac{3}{2}}E_0\circ E_1\big(E_2(\mathcal{F}_1(y^*)\mathcal{F}_1(x^*)e_1e_2)z\big)\\\label{inproofofcp}
&=& \tau^{-\frac{3}{2}}E_0\circ E_1\circ E_2\big(\mathcal{F}_1(y^*)\mathcal{F}_1(x^*)e_1e_2z\big).
\end{eqnarray}
Using the definition of the Fourier transform, we obtain the following:
\begin{eqnarray}\nonumber
\mathcal{F}_1(y^*)\mathcal{F}_1(x^*)e_1e_2z&=&\tau^{-1}\sum_{i_1,i_2}\lambda_{i_1}y^*e_2e_1\lambda^*_{i_1}\lambda_{i_2}x^*e_2e_1\lambda^*_{i_2}e_1e_2z\\\nonumber
&=&\tau^{-1}\sum_{i_1,i_2}\lambda_{i_1}y^*e_2e_1\lambda^*_{i_1}\lambda_{i_2}x^*e_2E_0(\lambda^*_{i_2})e_1e_2z\\\nonumber
&=&\sum_{i_1}\lambda_{i_1}y^*e_2e_1\lambda^*_{i_1}x^*e_2z\\\nonumber
&=&\sum_{i_1}\lambda_{i_1}y^* E_1(e_1\lambda^*_{i_1}x^*)e_2z.
\end{eqnarray}
It follows from \Cref{inproofofcp} that
\begin{eqnarray}\nonumber
E_0\circ E_1((x^* * y^*)z)
&=& \tau^{-\frac{3}{2}}E_0\circ E_1\circ E_2\big(\sum_{i_1}\lambda_{i_1}y^* E_1(e_1\lambda^*_{i_1}x^*)e_2z\big)\\\nonumber
&=& \tau^{-\frac{1}{2}}\sum_{i_1} E_0\circ E_1\big(\lambda_{i_1}y^* E_1(e_1\lambda^*_{i_1}x^*)z\big).
\end{eqnarray}
So we finally get,
\begin{eqnarray}\label{RHS}
\tau^{\frac{1}{2}}\big<z, x*y\big>&=&\sum_{i}E_0 \circ E_1\big(z^*E_1(x\lambda_{i}e_1)y\lambda^*_{i}\big).
\end{eqnarray}
Combining \Cref{LHS} and \Cref{RHS}, we get the desired identity.
\qed
\end{prf}
\begin{thm}\label{impformainthm6}
$\Delta$ is a $*$-preserving completely positive map.
\end{thm}
\begin{prf}
Using \Cref{main theorem}, we have 
\[
\Delta(z)= \tau^{-1}\,T^{-1}\circ E(\tau^{-1}e_2ze_3e_2).
\]
It is easy to verify that the map $z\rightarrow \tau^{-1}e_2ze_3e_2$ is a \(*\)-algebra homomorphism. Hence, \(\Delta\) is \(*\)-preserving. Since any \(*\)-algebra homomorphism between two \(C^*\)-algebras is completely positive, it follows that the map $z\rightarrow \tau^{-1}e_2ze_3e_2$ is completely positive. Moreover, since \(T^{-1}\) is a \(*\)-isomorphism, it is also completely positive. Therefore, as a composition of completely positive maps, \(\Delta\) is completely positive.
\qed
\end{prf}

\bigskip
Now, instead of starting with the inclusion of simple unital \( C^* \)-algebras \( B \subset A \), we consider the inclusion \( A \subset A_1 \). In this situation, we obtain a convolution product on the relative commutant \(A' \cap A_2\). We denote this convolution product by \(\star_1\), corresponding to the inclusion \(A \subset A_1\). Using \Cref{weak hypergroup structure}, the comultiplication and counit on \(A' \cap A_2\) are given by
\[
\big<\Delta_{A' \cap A_{2}}(w), w'\otimes w''\big> =\tau^{\frac{1}{2}} \big<w, w' *_1 w''\big>,
\]
\[
\varepsilon_{A' \cap A_{2}}(w) = \tau^{-1}\big<e_2, w\big>,
\]
for all \(w, w', w'' \in A' \cap A_2\). We now aim to relate this comultiplication and counit to those defined by Nikshych and Vainerman in \cite{NV} on \(A' \cap A_2\) (see also \cite{preprint}). The bilinear form
\[
\langle x, w \rangle_{\mathrm{NV}}
=
d\,\tau^{-2}\,
\operatorname{tr}(x\, e_2 e_1\, w),
\qquad 
x \in B' \cap A_1,\ w \in A' \cap A_2,
\]
where \( d \) denotes the constant number  \( \| {\mathrm{Ind}}_{\text{W}}(E_1|_{A' \cap A_1}) \|_2 \), defines a nondegenerate duality between \(B' \cap A_1\) and \(A' \cap A_2\) (\cite[Proposition 3.2]{preprint}).
\begin{dfn}[\cite{NV}]\label{dfnold}
Using the pairing $\big<,\big>_{\text{NV}}$, the comultiplication  $\Delta^{\text{NV}}_{A'\cap A_2}$ and the counit $\varepsilon^{\text{NV}}_{A'\cap A_2}$ on $A'\cap A_2$ are defined as follows:
$$\Delta^{\text{NV}}_{A'\cap A_2}:A'\cap A_2 \rightarrow (A'\cap A_2) \otimes (A'\cap A_2): \quad  \,\big<x_1x_2,w\big>_{\text{NV}}=\big<x_1,w_{(1)}\big>_{\text{NV}} \,\,\big<x_2,w_{(2)}\big>_{\text{NV}} ,$$
$$\varepsilon^{\text{NV}}_{A'\cap A_2}: A'\cap A_2 \rightarrow \bbc: \quad \varepsilon^{\text{NV}}_{A'\cap A_2}(w)=\big<1,w\big>_{\text{NV}} ,$$
for all $x,x_1,x_2 \in B'\cap A_1$ and $w \in A'\cap A_2$.
\end{dfn}

In order to verify the relation, we begin with the following results.
\begin{lmma}\label{reln between form and ip}
Let $x \in B'\cap A_1$ and $w \in A' \cap A_2$. Then
$$\big<x,w\big>_{\text{NV}}= d\tau^{-1}\big<w^*_{x}, w\big>,$$
for some \(w_x \in A' \cap A_2\) such that  $\mathcal{F}_1^{-1}(w_x) = \tau^{\frac{1}{2}} x.$
\end{lmma}
\begin{prf}
We have $xe_2\in B'\cap A_2$. By \cite[Lemma 3.1]{preprint}, $xe_2e_1\in (B'\cap A_2)e_1=(A'\cap A_2)e_1$. Thus, there exists $w_x \in A'\cap A_2$ such that $xe_2e_1= w_xe_1$. Therefore, we obtain the following:
\begin{eqnarray}\nonumber
\big<x,w\big>_{\text{NV}}&=&d\tau^{-2}\tr(w_xe_1w)\\\nonumber
&=&d\tau^{-2}(ww_xe_1)\\\nonumber
&=&d\tau^{-2}\tr\big(ww_xE^{B^{\prime}\cap A_2}_{A^{\prime}\cap A_2}(e_1)\big)\\\nonumber
&=&d\tau^{-1}\big<w^*_x, w\big>.
\end{eqnarray}
Furthermore, we have $\mathcal{F}^{-1}_1(w_x)=\tau^{-\frac{3}{2}}E_2(w_xe_1e_2)=\tau^{-\frac{1}{2}}E_2(xe_2)=\tau^{\frac{1}{2}}x.$
\qed
\end{prf}
\begin{thm}\label{main theorem 2}
The two comultiplications on $A'\cap A_2$ are related by the following relation:
$$\Delta_{A' \cap A_{2}}(w)=d\, \Delta^{\text{NV}}_{A'\cap A_2}(w), \quad \text{ for all } w\in A'\cap A_2. $$
\end{thm}
\begin{prf}
Let us assume that $\Delta^{\text{NV}}_{A'\cap A_2}(w)= w^{\text{NV}}_{(1)}\otimes w^{\text{NV}}_{(2)}$ and $\Delta_{A' \cap A_{2}}(w)= w_{(1)}\otimes w_{(2)}$. Then we have the following:
\begin{eqnarray}\nonumber
\big<x_1, w_{(1)}\big>_{\text{NV}}\,\big<x_2, w_{(2)}\big>_{\text{NV}}&=&d^2 \tau^{-2}\big<w^*_{x_1}, w_{(1)}\big>\,\big<w^*_{x_2}, w_{(2)}\big>\\\nonumber
&=&d^2 \tau^{-\frac{3}{2}}\big<w^*_{x_1} \star_{1} w^*_{x_2}, w\big>\\\nonumber
&=&d^2 \tau^{-\frac{3}{2}}\big<w^*_{x_2} \star w^*_{x_1}, w\big> \quad \text{( by \cite{BGPS}[Lemma 4.11]) }\\\nonumber
&=&d^2 \tau^{-\frac{3}{2}}\tr((w_{x_2} \star w_{x_1}) w)\quad \text{( by \Cref{3Proposition 3.8}) }\\\nonumber
&=&d^2 \tau^{-\frac{3}{2}}\tr(\mathcal{F}_1(\mathcal{F}^{-1}_1(w_{x_1})\mathcal{F}^{-1}_1(w_{x_2}))w).
\end{eqnarray}
Now using the definition of the Fourier transform and the bimodule property of the trace preseving conditional expectation $E^{B^{\prime}\cap A_2}_{A^{\prime}\cap A_2}$, we get the following:
\begin{eqnarray}\nonumber
\big<x_1, w_{(1)}\big>_{\text{NV}}\,\big<x_2, w_{(2)}\big>_{\text{NV}}
&=&d^2 \tau^{-3}\tr\big(E^{B^{\prime}\cap A_2}_{A^{\prime}\cap A_2}(\mathcal{F}^{-1}_1(w_{x_1})\mathcal{F}^{-1}_1(w_{x_2})e_2e_1w)\big)\\\nonumber
&=&d^2 \tau^{-3}\tr(\mathcal{F}^{-1}_1(w_{x_1})\mathcal{F}^{-1}_1(w_{x_2})e_2e_1w)\\\nonumber
&=&d^2 \tau^{-2}\tr(x_1x_2e_2e_1w)\\\nonumber
&=&d\, \big<x_1x_2, w\big>_{\text{NV}}\\\nonumber
&=&d \,\big<x_1, w^{{\text{NV}}}_{(1)}\big>_{\text{NV}}\,\big<x_2, w^{{\text{NV}}}_{(2)}\big>_{\text{NV}}
\end{eqnarray}
Thus we get $w_{(1)}\otimes w_{(2)}=d \, w^{\text{NV}}_{(1)}\otimes w^{\text{NV}}_{(2)}$. Which implies $\Delta_{A' \cap A_{2}}(w)=d  \,\Delta^{\text{NV}}_{A'\cap A_2}(w)$.
\qed
\end{prf}
\begin{ppsn}
The two counits defined on $A'\cap A_2$ are related as follows:
$$\varepsilon^{\text{NV}}_{A'\cap A_2}(w)=d\,\varepsilon_{A'\cap A_2}(w).$$
\end{ppsn}
\begin{prf}
From \Cref{dfnold}, we have the following:
\begin{eqnarray}\nonumber
 \varepsilon^{\text{NV}}_{A'\cap A_2}(w)&=&\big<1,w\big>_{\text{NV}}\\\nonumber
 &=& d\tau^{-2}\tr(e_2e_1w)\\\nonumber
 &=& d\tau^{-2}\tr(E^{B^{\prime}\cap A_2}_{A^{\prime}\cap A_2} (e_2e_1w))\\\nonumber
 &=& d\tau^{-1}\tr( e_2w)\\\nonumber
 &=& d\,\varepsilon_{A'\cap A_2}(w)
\end{eqnarray}
 \qed
\end{prf}
\begin{rmrk}\label{impincrlre}
By the nondegenerate duality $\langle \,,\, \rangle_{\mathrm{NV}}$, a coalgebra structure is induced on the second relative commutant \(B' \cap A_1\) as the dual of the coalgebra \(A' \cap A_2\) defined in \Cref{dfnold}. Let \(\Delta^{\mathrm{NV}}_{B' \cap A_1}\) and \(\varepsilon^{\mathrm{NV}}_{B' \cap A_1}\) denote the coproduct and counit arising from this duality. It is straightforward to verify that, $d \, \Delta^{\mathrm{NV}}_{B' \cap A_1} = \Delta$ and $\varepsilon^{\mathrm{NV}}_{B' \cap A_1} = d \, \varepsilon$, where \(\Delta\) and \(\varepsilon\) denote the coproduct and counit defined in \Cref{weak hypergroup structure}. Thus, the coalgebra structures as defined in \Cref{weak hypergroup structure} on \(B' \cap A_1\) arising from the inclusion \(B \subset A\), and on \(A' \cap A_2\) arising from the inclusion \(A \subset A_1\), are dual to each other via the nondegenerate pairing \(\langle \,,\, \rangle_{\mathrm{NV}}\). Moreover, a direct computation shows that these two coalgebra structures on each \(B' \cap A_1\) and \(A' \cap A_2\) are isomorphic.
\end{rmrk}

For a finite index $II_1$ factors \(N \subset M\), Bhattacharjee et al.\ established in \cite{BCG} the existence of a universal Hopf \(*\)-algebra associated to the inclusion, referred to as the quantum Galois group \(Q\!\operatorname{Gal}(N \subset M)\). This leads to the following natural question.

\medskip

\noindent\textbf{Question.} Is there a relationship between the quantum Galois group $Q\!\operatorname{Gal}(N \subset M)$ and the Hopf algebra constructed in \Cref{Takeuchi} ?

\section{Weak quantum Hypergroup structure on the second relative commutant}

The primary objective of this section is to demonstrate that, given an inclusion \(B \subset A\) of simple unital \(C^*\)-algebras equipped with a conditional expectation of index-finite type, the second relative commutant \(B' \cap A_1\) naturally acquires a structure, defined below, that we refer to as a \emph{weak quantum hypergroup}. This notion may be viewed as a relaxation of the quantum hypergroup concept introduced by Chapovsky and Vainerman in \cite{CV}. Furthermore, we show that in the case where the inclusion is irreducible, the resulting structure upgrades to a genuine quantum hypergroup. Finally, we show that the weak quantum hypergroup $B' \cap A_1$ admits both left-invariant and right-invariant normalized measures, as well as a Haar integral.

	\begin{dfn}\label{weak hypergroup}
	Let $(A,\cdot,1,*)$ be a unital $C^*$-algebra. We say that $(A, \Delta, \varepsilon,{\small \texttt{\#}})$ defines a \textbf{weak quantum hypergroup} structure on $(A,\cdot,1,*)$ if the following conditions hold:
	\begin{itemize}
	\item
	$(A, \Delta, \varepsilon,{\small \texttt{\#}} )$ is a ${\small \texttt{\#}}$-coalgebra with a counit $\varepsilon$; that is, $\Delta:A\rightarrow A\otimes A$ and $\varepsilon: A\rightarrow \bbc$ are linear maps, and ${\small \texttt{\#}}: A \rightarrow A $ is an antilinear map satisfying:
	\begin{enumerate}
	\item[$(i)$] $(\Delta \otimes \mathrm{id})\circ \Delta= (\mathrm{id}\otimes \Delta) \circ \Delta,$
	\item[$(ii)$] $(\varepsilon\otimes \mathrm{id}) \circ \Delta= (\mathrm{id}\otimes \varepsilon) \circ \Delta = \mathrm{id},$
	\item[$(iii)$] $\Delta \circ {\small \texttt{\#}} = \Pi \circ ({\small \texttt{\#}} \otimes {\small \texttt{\#}})\circ \Delta$,\quad where $\Pi: A \otimes A \rightarrow A \otimes A$ is the flip map,
	\item[$(iv)$] ${\small \texttt{\#}} \circ {\small \texttt{\#}} =\mathrm{id}$.
	\end{enumerate}
	\item
	The map $\Delta$ is positive.
	\item The following identities hold:
	\begin{enumerate}
	\item[$(i)$]  $(x\cdot y)^{{\small \texttt{\#}}}=x^{{\small \texttt{\#}}} \cdot y^{{\small \texttt{\#}}},$ \quad for all $x,y \in A$, 
	\item [$(ii)$] $\Delta \circ * = (* \otimes *) \circ \Delta$,
	\item [$(iii)$] ${\small \texttt{\#}} \circ *= * \circ {\small \texttt{\#}}$.
	\end{enumerate}
	\item \textbf{Weak multiplicativity of the counit} holds, i.e., for all $x,y, z \in A$, we have
	$$\varepsilon(xyz)= \varepsilon(x y_{(1)}) \varepsilon(y_{(2)} z) = \varepsilon(x y_{(2)}) \varepsilon(y_{(1)} z).$$
	\item \textbf{Weak comultiplicativity of the unit} holds, i.e.,
	$$(\mathrm{id}\otimes \Delta)\circ\Delta(1)=(\Delta(1)\otimes 1)(1\otimes \Delta(1))=(1\otimes \Delta(1))(\Delta(1)\otimes 1).$$
	\end{itemize}
	\end{dfn}
	
As in the case of weak Hopf algebras, we define the following notions for a weak quantum hypergroup.	
	
\begin{dfn}
For any weak quantum hypergroup \(A\), we define the \emph{target} and \emph{source} maps, denoted by \(\varepsilon^t\) and \(\varepsilon^s\), respectively, on \(A\) as follows:
\[
\varepsilon^t(x) = \varepsilon(1_{(1)}x)\,1_{(2)}, \quad \text{for all } x \in A,
\]
\[
\varepsilon^s(x) = 1_{(1)}\,\varepsilon(x1_{(2)}), \quad \text{for all } x \in A.
\]
\end{dfn}

\begin{dfn}
A functional $\phi \in A^*$ is called a left-invariant measure on $A$ if
\[
(\mathrm{id}\otimes \phi) \circ \Delta = (\varepsilon^t \otimes \phi) \circ \Delta.
\]
The measure $\phi$ is said to be normalized if $
(\mathrm{id} \otimes \phi) \circ \Delta(1) = 1.$

Similarly, a functional $\psi \in A^*$ is called a right-invariant measure on $A$ if
\[
(\psi \otimes \mathrm{id}) \circ \Delta = (\psi \otimes \varepsilon^s) \circ \Delta.
\]
The measure $\psi$ is said to be normalized if $
(\psi \otimes \mathrm{id}) \circ \Delta(1) = 1.$
\end{dfn}
\begin{dfn}
A left integral is an element $x \in A$ such that $a x = \varepsilon^t(a) x $ for all $a \in A.$ Similarly, a right integral is an element $y \in A$ such that $y a = y \varepsilon^s(a)$ for all $a \in A.$

An element $h \in A$ is called a Haar integral if it is both a left and right integral and satisfies $\varepsilon^t(h) = \varepsilon^s(h) = 1.$
\end{dfn}

We now show that for an inclusion of simple unital $C^*$-algebras with a conditional expectation of finite Watatani index, the second relative commutant $B' \cap A_1$ carries the structure of a weak quantum hypergroup as defined in \Cref{weak hypergroup}. We note that Bischoff et al. \cite{BVG} introduced a notion of compact hypergroups associated with irreducible local discrete subfactors. Our approach is more general in the sense that the inclusion under consideration need not be irreducible. To establish this structure, we define the following maps on $B' \cap A_1$.

\begin{dfn}\label{weak hypergroup structure2}
Let $B\subset A$ be an inclusion of simple unital $C^*$-algebras with a conditional expectation of index-finite type. On the second relative commutant $B'\cap A_1$, we take $\Delta$, $\varepsilon$ as defined in \Cref{weak hypergroup structure} and we define ${\small \texttt{\#}}$ as follows:
$$ x^{{\small \texttt{\#}}}= r^+_1(x^*), \quad \text { for any }x\in B'\cap A_1.$$
\end{dfn}

\begin{ppsn}\label{impformainthm2}
 For any \(x \in B' \cap A_1\), we have
$$\Delta(x^{{\small \texttt{\#}}})=x^{{\small \texttt{\#}}}_{(2)} \otimes x^{{\small \texttt{\#}}}_{(1)
}.$$
\end{ppsn}
\begin{prf}
Let \(x, y, z \in B' \cap A_1\). Then we have the following:
\begin{eqnarray}\nonumber
\big<x^{{\small \texttt{\#}}}_{(2)} \otimes x^{{\small \texttt{\#}}}_{(1)
}, y\otimes z \big>
= \tr\big(r^+_1(x_{
(2)})y\big)\,\tr\big(r^+_1(x_{(1)})z\big).
\end{eqnarray}
From \Cref{aboutreflectionoperator}, we know that $r^+_1$ is involutive and trace preserving. Hence we get,
\begin{eqnarray}\nonumber
\big<x^{{\small \texttt{\#}}}_{(2)} \otimes x^{{\small \texttt{\#}}}_{(1)
}, y\otimes z \big>&=& = \tr\big(x_{
(2)}r^+_1( y)\big)\,\tr\big(x_{(1)} r^+_1(z)\big)\\\nonumber
&=& \big<x^*_{(2)} ,r^+_1(y) \big>\big< x^{*}_{(1)
}, r^+_1(z) \big>\\\nonumber
&=&\tau^{\frac{1}{2}}\big<x^*,r^+_1(z) \star r^+_1(y)\big>.
\end{eqnarray}
From \cite[Proposition 3.9]{BGS}, we know that $r^+_1(z) \star r^+_1(y)= r^+_1(y \star z)$. So we get the following:
\begin{eqnarray}\nonumber
\big<x^{{\small \texttt{\#}}}_{(2)} \otimes x^{{\small \texttt{\#}}}_{(1)
}, y\otimes z \big>&=&\tau^{\frac{1}{2}}\big<x^*,r^+_1(y \star z)\big>\\\nonumber
&=&\tau^{\frac{1}{2}}\tr\big(r^+_1(x) (y\star z)\big)\\\nonumber
&=& \tau^{\frac{1}{2}}\big<x^{{\small \texttt{\#}}} , y\star z \big>\\\nonumber
&=&\big<\Delta(x^{{\small \texttt{\#}}}),y\otimes z\big>.
\end{eqnarray}
Hence we are done. 
\qed
\end{prf}

\begin{ppsn}\label{impformainthm3}
We have the followings:
\begin{enumerate}
\item[$(i)$] ${\small \texttt{\#}}: B'\cap A_1 \rightarrow B'\cap A_1$ is antilinear.
\item[$(ii)$] For any $x \in B'\cap A_1$, $(x^{{\small \texttt{\#}}})^{{\small \texttt{\#}}}=x.$
\item[$(iii)$] For any $x,y \in B'\cap A_1$, $(x\cdot y)^{{\small \texttt{\#}}}=x^{{\small \texttt{\#}}} \cdot y^{{\small \texttt{\#}}}.$
\item[$(iv)$] For any $x \in B'\cap A_1$, $(x^*)^{{\small \texttt{\#}}}=(x^{{\small \texttt{\#}}})^*.$
\end{enumerate}
\end{ppsn}
\begin{prf}
Since \( r_1^+ \) is a linear and \(*\)-preserving antihomomorphism, the preceding statements follow immediately.
\qed
\end{prf}

We now proceed to establish the weak multiplicativity of the counit. Before entering into the main part of the argument, we record a number of auxiliary observations that will be essential for our proof.
\begin{lmma}\label{impfor e(xyz)}
For any $x,y,z \in B'\cap A_1$, we have the followings:
\begin{enumerate}
\item[$(i)$] $\big<z \,, \,(xe_1)\star y\big>=\tau^{-\frac{1}{2}}\tr\big(E_1(xe_1)yz^*\big).$
\item[$(ii)$] $\varepsilon^t(x)= \tau^{-1}E_1(xe_1).$
\item[$(iii)$] $\varepsilon(x_{(1)}y)x_{(2)}=x\varepsilon^t(y).$
\end{enumerate}
\end{lmma}
\begin{prf}
\begin{enumerate}[leftmargin=*]
\item[$(i)$] Using \Cref{3Proposition 3.8}, we have the following:
\begin{eqnarray}\label{impepsilont}
\big<z \,, \,(xe_1)\star y\big>&=& \overline{\tr(z\,((e_1x^*)\star y^*))}.
\end{eqnarray}
From the definition of the convolution product, it follows that:
\begin{eqnarray}\nonumber
\tr(z\,((e_1x^*)\star y^*))&=& \tau^{-\frac{3}{2}}\tr\big(z E_2(\mathcal{F}(y^*)\mathcal{F}(e_1x^*)e_1e_2)\big)\\\nonumber
&=&\tau^{-\frac{3}{2}}\tr\big(z\,\mathcal{F}(y^*)\mathcal{F}(e_1x^*)e_1e_2\big)\\\nonumber
&=& \tau^{-\frac{5}{2}}\sum_{i_1, i_2}\tr(\lambda_{i_1} e_1x^*e_2e_1\lambda^*_{i_1}e_1e_2z \lambda_{i_2} y^*e_2e_1\lambda^*_{i_2})\\\nonumber
&=& \tau^{-\frac{3}{2}}\sum_{i_1}\tr( e_1x^*e_2z \lambda_{i_2} y^*e_2e_1\lambda^*_{i_2})\\\nonumber
&=& \tau^{-\frac{3}{2}}\sum_{i_1}\tr( x^*e_2z \lambda_{i_2} y^*e_2E_0(\lambda^*_{i_2})e_1)\\\nonumber
&=& \tau^{-\frac{3}{2}}\tr( x^*e_2z y^*e_2e_1).
\end{eqnarray}
Combining the above with \Cref{impepsilont}, we arrive at
\begin{eqnarray}\nonumber
\big<z \,, \,(xe_1)\star y\big>&=& \tau^{-\frac{3}{2}}\tr(e_1e_2yz^*e_2x)\\\nonumber
&=&\tau^{-\frac{3}{2}} E_0\circ E_1 \circ E_2(e_1E_1(yz^*)e_2x)\\\nonumber
&=&\tau^{-\frac{1}{2}} E_0\circ E_1(xe_1E_1(yz^*))\\\nonumber
&=&\tau^{-\frac{1}{2}} E_0(E_1(xe_1)E_1(yz^*))\\\nonumber
&=& \tau^{-\frac{1}{2}}\tr\big(E_1(xe_1)yz^*\big).
\end{eqnarray}
\item[$(ii)$] It suffices to show that, for any $y\in B'\cap A_1$
\begin{eqnarray}\nonumber
\big< y,\varepsilon(1_{(1)}x)\,1_{(2)}\big>= \tau^{-1}\big<y,  E_1(xe_1)\big>. 
\end{eqnarray}
Now using \Cref{weak hypergroup structure2} and above item $(i)$, we get the following:
\begin{eqnarray}\nonumber
\big<y,  \varepsilon(1_{(1)}x)\,1_{(2)}\big>&=& \tau^{-1}\big<e_1,1_{(1)}x\big>\,\big<y, 1_{(2)}\big>\\\nonumber
&=&\tau^{-1}\big<1^*_{(1)}, xe_1\big>\,\big<1^*_{(2)},y^*\big>\\\nonumber
&=&\tau^{-\frac{1}{2}}\big<1, (xe_1) \star y^*\big>\\\nonumber
&=&\tau^{-1}\,\tr\big(E_1(xe_1)y^*\big)\\\nonumber
&=&\tau^{-1}\big<y,  E_1(xe_1)\big>.
\end{eqnarray}
Hence we are done.
\item[$(iii)$] For any $z \in B'\cap A_1$, we have the following:
\begin{eqnarray}\nonumber
\big<z,\varepsilon(x_{(1)}y)x_{(2)}\big>&=& \tau^{-1}\big<e_1,x_{(1)}y\big>\,\big<z,x_{(2)}\big>\\\nonumber
&=& \tau^{-1}\big<x^*_{(1)},ye_1\big>\,\big<x^*_{(2)},z^*\big>\\\nonumber
&=& \tau^{-\frac{1}{2}}\big<x^*,(ye_1)\star z^*\big>\\\nonumber
&=&\tau^{-1}\tr\big(E_1(ye_1)z^*x\big)\\\nonumber
&=& \big<z, x\varepsilon^t(y)\big>.
\end{eqnarray}
So the proof is complete.
\qed
\end{enumerate}
\end{prf}
\begin{ppsn}\label{impformainthm4}
For any $x,y, z \in B'\cap A_1$, we have
	$$\varepsilon(xyz)= \varepsilon(x y_{(1)}) \varepsilon(y_{(2)} z) = \varepsilon(x y_{(2)}) \varepsilon(y_{(1)} z).$$
\end{ppsn}
\begin{prf}
Using \Cref{impfor e(xyz)}, we get the following:
\begin{eqnarray}\nonumber
\varepsilon(x y_{(2)}) \varepsilon(y_{(1)} z)&=&\tau^{-1} \varepsilon(x yE_1(ze_1))\\\nonumber
&=& \tau^{-2}\tr(e_1x yE_1(ze_1))\\\nonumber
&=& \tau^{-2}\tr(E_1(ze_1) e_1x y)\\\nonumber
&=& \tau^{-1}\tr(ze_1x y)\\\nonumber
&=& \varepsilon(xyz).
\end{eqnarray}
Since $\Delta$ is $*$-preserving, we know that $\varepsilon(x)=\overline{\varepsilon(x^*)}$. Thus we have the following:
\begin{eqnarray}\nonumber
\varepsilon(xyz)&=& \overline{\varepsilon(z^*y^*x^*)}\\\nonumber
&=& \overline{\varepsilon(z^*y^*_{(2)})\varepsilon(y^*_{(1)}x^*)}\\\nonumber
&=& \varepsilon(x y_{(1)}) \varepsilon(y_{(2)} z).
\end{eqnarray}
Hence we are done.
\qed
\end{prf}

We now turn to the verification of the weak comultiplicativity of the unit. 
This property is a crucial ingredient in demonstrating that the second relative 
commutant \(B' \cap A_1\), associated with the inclusion \(B \subset A\), 
indeed carries the structure of a weak quantum hypergroup. 
In order to prepare for the main argument, we first collect a number of 
preliminary observations that will serve as essential tools in the subsequent 
analysis.

\begin{lmma}\label{impfordelt(1)otherthings1}
Let $\{\gamma_i\}$ be a quasi-basis of $E_0|_{B'\cap A}$, then we have the following:
$$\Delta(1)=  \sum_i r^+_1(\gamma^*_i) \otimes \gamma_i. $$
\end{lmma}
\begin{prf}
For any $x,y \in B'\cap A_1$, we have
\begin{eqnarray}\nonumber
\big<\Delta(1), x\otimes y\big>&=& \tau^{\frac{1}{2}}\big<1, x\star y\big>\\\nonumber
&=&\tau^{-1} \tr(E_2(\mathcal{F}(y)\mathcal{F}(x)e_1e_2))\\\nonumber
&=&\tau^{-2} \tr(\sum_{{i_1},i_2} \lambda_{i_1} xe_2e_1\lambda^*_{i_1}e_1e_2\lambda_{i_2}ye_2e_1\lambda^*_{i_2} )\\\nonumber
&=&\tau^{-1} \tr(\sum_{i_2} xe_2\lambda_{i_2}ye_2e_1\lambda^*_{i_2} )\\\nonumber
&=&\tau^{-1} \tr(\sum_{i_2} x\lambda_{i_2}E_1(y)e_2e_1\lambda^*_{i_2} )\\\nonumber
&=& \tr(\sum_{i_2} x\lambda_{i_2}E_1(y)e_1\lambda^*_{i_2} ).
\end{eqnarray}
Using \Cref{reflection operator formula}, we have the following:
\begin{eqnarray}\nonumber
r^+_1(E_1(y))=\tau^{-1}\sum_{i}E_1(e_1 \lambda_iE_1(y))e_1\lambda^*_i=\sum_{i} \lambda_i E_1(y)e_1\lambda^*_i.
\end{eqnarray}
Consequently, we obtain the following:
\begin{eqnarray}\label{valuedelta(1)1}
\big<\Delta(1), x\otimes y\big>=\tr\big(x\,r^+_1(E_1(y))\big).
\end{eqnarray}
We further observe that
\begin{eqnarray}\nonumber
\big<\sum_i r^+_1(\gamma^*_i) \otimes \gamma_i,x \otimes y \big>&=&\sum_{i}\tr(r^+_1(\gamma_i)x)\,\tr(\gamma^*_iy)\\\nonumber
&=& \sum_{i}\tr(r^+_1(\gamma_i)x)\,E_0(\gamma^*_iE_1(y))\\\nonumber
&=& \tr\big(r^+_1(\sum_{i}\gamma_iE_0(\gamma^*_iE_1(y)))x\big)\\\nonumber
&=& \tr\big(r^+_1(E_1(y))x\big)\\\label{valuedelta(1)2}
&=&\tr\big(x\,r^+_1(E_1(y))\big).
\end{eqnarray}
So finally from \Cref{valuedelta(1)1} and \Cref{valuedelta(1)2}, we are done.
\end{prf}
\begin{lmma}\label{impfordelt(1)otherthings2}
For any $x\in B'\cap A_1$ and $v\in B'\cap A$, we have
$$\Delta(xv)=\Delta(x)(v\otimes 1).$$
\end{lmma}
\begin{prf}
For any $y,z \in B'\cap A_1$, we have the following:
\begin{eqnarray}\nonumber
\big<\Delta(xv), y\otimes z\big>&=&\tau^{\frac{1}{2}} \big<xv,y \star z\big>\\\nonumber
&=& \tau^{-1}\tr(v^*x^*E_2(\mathcal{F}(z)\mathcal{F}(y)e_1e_2))\\\nonumber
&=& \tau^{-1}\tr(v^*x^*\mathcal{F}(z)\mathcal{F}(y)e_1e_2)\\\nonumber
&=& \tau^{-2}\sum_{i_1,i_2}\tr(\lambda_{i_1}ye_2e_1\lambda^*_{i_1}e_1e_2v^*x^*\lambda_{i_2}ze_2e_1\lambda^*_{i_2})\\\label{impfordelta(xv)1}
&=& \tau^{-1}\sum_{i_2}\tr(yv^*e_2x^*\lambda_{i_2}ze_2e_1\lambda^*_{i_2}).
\end{eqnarray}
Also we can observe the following:
\begin{eqnarray}\nonumber
\big< x_{(1)}v, y\big>\,\big< x_{(2)}, z\big> &=&\big< x_{(1)}, yv^*\big>\,\big< x_{(2)}, z\big>\\\nonumber
&=& \tau^{\frac{1}{2}}\big< x, (yv^*)\star z\big>\\\nonumber
&=& \tau^{-1}\tr(x^*E_2(\mathcal{F}(z)\mathcal{F}(yv^*)e_1e_2))\\\nonumber
&=& \tau^{-1}\tr(x^*\mathcal{F}(z)\mathcal{F}(yv^*)e_1e_2)\\\nonumber
&=& \tau^{-2}\sum_{i_1,i_2}\tr(\lambda_{i_1}yv^*e_2e_1\lambda^*_{i_1}e_1e_2x^*\lambda_{i_2}ze_2e_1\lambda^*_{i_2})\\\label{impfordelta(xv)2}
&=& \tau^{-1}\sum_{i_2}\tr(yv^*e_2x^*\lambda_{i_2}ze_2e_1\lambda^*_{i_2}).
\end{eqnarray}
So finally from \Cref{impfordelta(xv)1} and \Cref{impfordelta(xv)2}, we are done.
\qed
\end{prf}
\begin{lmma}\label{impfordelt(1)otherthings3}
We have, $r^+_1(B'\cap A)= A'\cap A_1.$
\end{lmma}
\begin{prf}
Let $v\in B'\cap A$, then we have the following:
\begin{eqnarray}\nonumber
r^+_1(v)=\tau^{-1}\sum_{i}E_1(e_1 \lambda_iv)e_1\lambda^*_i=\sum_{i} \lambda_i ve_1\lambda^*_i&=&\tau^{-1} E^{B^{\prime}\cap A_1}_{A^{\prime}\cap A_1}(ve_1).
\end{eqnarray}
Hence we get that, $r^+_1(B'\cap A)\subset A'\cap A_1$. Also for any $u\in A'\cap A_1$, we get the following:
\begin{eqnarray}\nonumber
r^+_1(u)&=&\tau^{-1}\sum_{i}E_1(e_1 \lambda_iu)e_1\lambda^*_i\\\nonumber
&=& \tau^{-1}\sum_{i}E_1(e_1u \lambda_i)e_1\lambda^*_i\\\nonumber
&=& \tau^{-1}\sum_{i}E_1(e_1u) \lambda_ie_1\lambda^*_i\\\nonumber
&=& \tau^{-1}\sum_{i}E_1(e_1u).
\end{eqnarray}
Hence we get, $r^+_1(A'\cap A_1)\subset B'\cap A$. Since $r^+_1$ is involutive, we get the desired result.
\qed
\end{prf}
\begin{ppsn}\label{impformainthm5} 
The following holds,
$$(\mathrm{id}\otimes \Delta)\circ\Delta(1)=(\Delta(1)\otimes 1)(1\otimes \Delta(1))=(1\otimes \Delta(1))(\Delta(1)\otimes 1).$$
\end{ppsn}
\begin{prf}
Using \Cref{impfordelt(1)otherthings1} and \Cref{impfordelt(1)otherthings3}, we get the following:
\begin{eqnarray}\nonumber
(\Delta(1)\otimes 1)(1\otimes \Delta(1))&=& (\sum_i r^+_1(\gamma^*_i) \otimes \gamma_i\otimes 1)(\sum_j 1\otimes r^+_1(\gamma^*_j) \otimes \gamma_j)\\\nonumber
&=& \sum_{i,j} r^+_1(\gamma^*_i) \otimes \gamma_i\, r^+_1(\gamma^*_j)\otimes\gamma_j\\\label{impforppsndelta(1)others1}
&=&\sum_{i,j} r^+_1(\gamma^*_i) \otimes  r^+_1(\gamma^*_j)\, \gamma_i\otimes\gamma_j\\\nonumber
&=&(1\otimes \Delta(1))(\Delta(1)\otimes 1).
\end{eqnarray}
By applying \Cref{impfordelt(1)otherthings1} and \Cref{impfordelt(1)otherthings2}, we obtain:
\begin{eqnarray}\nonumber
(\mathrm{id}\otimes \Delta)\circ\Delta(1)
&=& \sum_ir^+_1(\gamma^*_i)\otimes \Delta(1 \gamma_i)\\\label{impforppsndelta(1)others2}
&=& \sum_{i,j}r^+_1(\gamma^*_i)\otimes r^+_1(\gamma^*_j)\,\gamma_i \otimes \gamma_j.
\end{eqnarray}
Finally, by \Cref{impforppsndelta(1)others1} and \Cref{impforppsndelta(1)others2}, we obtain the desired result.
\qed
\end{prf}

As a consequence of \Cref{impformainthm1}, \Cref{impformainthm6}, \Cref{impformainthm2}, 
\Cref{impformainthm3}, \Cref{impformainthm4}, and \Cref{impformainthm5}, we can now 
synthesize the structural properties established thus far. These results collectively ensure that the coproduct, counit, and the ${\small \texttt{\#}}$-operation on the second relative commutant interact in a manner fully compatible with the axioms of a weak quantum hypergroup. We therefore obtain the following fundamental theorem:
\begin{thm}\label{imp for Theorem B}
Let $B\subset A$ be an inclusion of simple unital $C^*$-algebras with a conditional expectation of index-finite type. If $\Delta$, $\varepsilon$ and ${\small \texttt{\#}}$ are as defined in \Cref{weak hypergroup structure2}. Then $(B'\cap A_1,\Delta, \varepsilon,{\small \texttt{\#}})$ becomes a weak quantum hypergroup.
\end{thm}




We recall that the pairing \(\langle \,,\, \rangle_{\mathrm{NV}}\) provides a nondegenerate duality between the relative commutants \(B' \cap A_1\) and \(A' \cap A_2\). This leads to the following consequence.

\begin{thm}
Let \(B \subset A\) be an inclusion of simple unital \(C^*\)-algebras with a conditional expectation of index-finite type. Then \(A' \cap A_2\) carries a weak quantum hypergroup structure induced by
this duality.
\end{thm}

\begin{proof}
By \Cref{imp for Theorem B}, the quadruple 
\((B' \cap A_1, \Delta, \varepsilon, \texttt{\#})\) forms a weak quantum hypergroup. 
Now consider the coproduct and counit on \(B' \cap A_1\) as in \Cref{dfnold}. Moreover, \Cref{impincrlre} yields $d\, \Delta^{\mathrm{NV}}_{B' \cap A_1} = \Delta$, $\varepsilon^{\mathrm{NV}}_{B' \cap A_1} = d\, \varepsilon,$ and therefore \((B' \cap A_1, \Delta^{\mathrm{NV}}_{B' \cap A_1}, \varepsilon^{\mathrm{NV}}_{B' \cap A_1}, \texttt{\#})\) is also a weak quantum hypergroup. Next, we consider the coproduct and counit on \(A' \cap A_2\) induced by the duality pairing \(\langle \cdot, \cdot \rangle_{\mathrm{NV}}\), making it the dual of \((B' \cap A_1, \Delta^{\mathrm{NV}}_{B' \cap A_1}, \varepsilon^{\mathrm{NV}}_{B' \cap A_1})\). By \Cref{main theorem 2}, we have $d \, \Delta^{\mathrm{NV}}_{A' \cap A_2} = \Delta_{A' \cap A_2},$ where \(\Delta_{A' \cap A_2}\) is the coproduct introduced in 
\Cref{weak hypergroup structure}, arising from the inclusion \(A \subset A_1\). Hence by \Cref{impformainthm6}, \(\Delta^{\mathrm{NV}}_{A' \cap A_2}\) is completely positive. We equip \(A' \cap A_2\) with a \(*\)-structure determined by
\[
\langle x, w^* \rangle_{\mathrm{NV}}
=
\overline{\langle x^{\texttt{\#}}, w \rangle}_{\mathrm{NV}},
\qquad x \in B' \cap A_1,\ \ w \in A' \cap A_2,
\]
and one verifies that this coincides with the usual involution on \(A' \cap A_2\). Furthermore, the map \(\texttt{\#} \colon A' \cap A_2 \to A' \cap A_2\) is defined by
\[
\langle x, w^{\texttt{\#}} \rangle_{\mathrm{NV}}
=
\overline{\langle x^*, w \rangle}_{\mathrm{NV}},
\qquad x \in B' \cap A_1,\ \ w \in A' \cap A_2.
\]
With these structures, the algebraic object 
\((A' \cap A_2, \Delta_{A' \cap A_2}, \varepsilon_{A' \cap A_2}, \texttt{\#})\) satisfies all of the axioms of a weak quantum hypergroup. 
This completes the proof.
\end{proof}

\begin{rmrk}
The algebra \(A' \cap A_2\) also carries a weak quantum hypergroup structure by \Cref{weak hypergroup structure2}, arising from the inclusion \(A \subset A_1\). In this setting, the coproduct and counit satisfy $d\, \Delta^{\mathrm{NV}}_{A' \cap A_2} = \Delta_{A' \cap A_2}$ and $\varepsilon^{\mathrm{NV}}_{A' \cap A_2} = d\, \varepsilon_{A' \cap A_2}.$ For the weak quantum hypergroup structure on \(A' \cap A_2\) coming from \Cref{weak hypergroup structure2}, the ${\texttt{\#}}$ is given by $w^{\texttt{\#}} := (r_1^+)^{A \subset A_1}(w^*).
$ By \cite[Proposition~3.20]{BGPS}, we have \((r_1^+)^{A \subset A_1} = r_1^-\). Moreover, by \cite[Lemma~4.3]{preprint}, $\langle x, r_1^-(w) \rangle_{\mathrm{NV}}
=
\langle r_1^+(x), w \rangle_{\mathrm{NV}},$ for $ x \in B' \cap A_1,\ w \in A' \cap A_2.$ Using these identities, it follows that the $\texttt{\#}$-operation defined in \Cref{weak hypergroup structure2} coincides with the one arising from the duality pairing \(\langle \cdot, \cdot \rangle_{\mathrm{NV}}\).
\end{rmrk}

If \(B' \cap A = \mathbb{C}\), then by \Cref{impfordelt(1)otherthings1} we have 
\(\Delta(1) = 1 \otimes 1\). Hence, from \Cref{impformainthm4} it follows that $\varepsilon(xy) = \varepsilon(x)\varepsilon(y)$ for all $x,y \in B' \cap A_1.$ Therefore, we obtain the following:

\begin{thm}\label{forirr}
Let $B\subset A$ be an irreducible inclusion of simple unital $C^*$-algebras with a conditional expectation of index-finite type. Then $(B'\cap A_1,\Delta, \varepsilon,{\small \texttt{\#}})$ becomes a quantum hypergroup.
\end{thm}
\begin{crlre}
Let \(B \subset A\) be an irreducible inclusion of simple unital \(C^*\)-algebras with a conditional expectation of index-finite type. Then \(A' \cap A_2\) carries a quantum hypergroup structure induced by the duality \(\langle \,,\, \rangle_{\mathrm{NV}}\).
\end{crlre}
\begin{rmrk}
In the case where the inclusion $B\subset A$ has depth \(2\) and $E_1|_{A'\cap A_1}$ has a scalar Watatani index, \(B' \cap A_1\) has the structure of a weak Kac algebra (see \cite{NV,preprint}). In this situation, the coproduct is a \(*\)-homomorphism and therefore completely positive, and all the axioms of a weak quantum hypergroup are automatically satisfied. Consequently, in the depth \(2\) case when $E_1|_{A'\cap A_1}$ has a scalar Watatani index  one recovers the structure of a weak quantum hypergroup in the sense of weak Kac algebras. Thus, the notion of a weak quantum hypergroup provides a more general framework.

It is important to note that \(B' \cap A_1\) admits many similar algebraic structures that satisfy all the axioms of a weak quantum hypergroup except the anti comultiplicativity of \({\small\texttt{\#}}\). For example, starting with the inclusion \(B \subset A_1\), one obtains a coproduct \(\Delta_3\) on \(B' \cap A_3\), and \(\Delta_3(B' \cap A_1) \subset (B' \cap A_3) \otimes (B' \cap A_1).\) Define \(\Delta_{D_1} = (E_2^3 \otimes \mathrm{id}) \circ \Delta_3\big|_{B'\cap A_1}, \ \varepsilon_{D_1} = \varepsilon_3\big|_{B'\cap A_1}, \ x^{\small\texttt{\#}} = r_1^+(x^*)\), where \(E_2^3 = E_2 \circ E_3\). This yields a structure on \(B' \cap A_1\) satisfying every weak quantum hypergroup axiom except \(\Delta \circ \texttt{\#} = \Pi \circ (\texttt{\#} \otimes \texttt{\#}) \circ \Delta\); instead, it satisfies \(\Delta_{D_1} \circ \texttt{\#} = (\texttt{\#} \otimes \texttt{\#}) \circ \Delta_{D_1}.\) Continuing in this manner along the tower of basic constructions yields infinitely many such structures, each sharing the same deviation from the anti comultiplicativity of \({\small\texttt{\#}}\).
\end{rmrk}

We now show that the weak quantum hypergroup \(B' \cap A_1\) appearing in 
\Cref{imp for Theorem B} carries both left-invariant and right-invariant normalized measures. In preparation for this result, we establish a number of preliminary observations that will be used throughout the proof.

\begin{lmma}\label{impforrightinvariantmeasure2}
For any $x\in B'\cap A_1$, we have the following:
$$\varepsilon^s(x)=\tau^{-1}r^+_1(E_1(e_1x)).$$
\end{lmma}
\begin{prf}
For \(y \in B' \cap A_1\), we obtain
\begin{eqnarray}\nonumber
\big<y,\, 1_{(1)}\varepsilon(x1_{(2)})\big> &=& \tau^{-1} \big<y,\, 1_{(1)}\big>\,\big<e_1,\, x1_{(2)}\big>\\\nonumber
&=& \tau^{-1}\tr(1_{(1)}y^*)\,\tr(1_{(2)}e_1 x)\\\nonumber
&=& \tau^{-1}\tr\big(r^+_1(1_{(1)})r^+_1(y^*)\big)\,\tr\big(r^+_1( 1_{(2)})r^+_1( e_1 x)\big)\quad \text{ (by \Cref{aboutreflectionoperator})}\\\nonumber
&=& \tau^{-1} \big<1^{{\small \texttt{\#}}}_{(1)},\,r^+_1(y^*) \big>\,\big< 1^{{\small \texttt{\#}}}_{(2)}, \, r^+_1( e_1 x)\big>\\\label{impforinvariantmeasure1}
&=& \tau^{-\frac{1}{2}} \big<1,\,r^+_1( e_1 x) \star r^+_1(y^*) \big> \quad \text{ (by \Cref{impformainthm2})} .
\end{eqnarray}
Applying \Cref{reflection operator formula}, we compute that
\begin{eqnarray}\nonumber
r^+_1(e_1x)=\tau^{-1}\sum_{i}E_1(e_1 \lambda_i e_1x)e_1\lambda^*_i= \tau^{-1}\sum_{i}E_1(E_0(\lambda_i)e_1x)e_1\lambda^*_i =\tau^{-1}E_1(e_1x)e_1.
\end{eqnarray}
Then, from \Cref{impforinvariantmeasure1}, we obtain
\begin{eqnarray}\nonumber
\big<y,\, 1_{(1)}\varepsilon(x1_{(2)})\big> &=&\tau^{-\frac{3}{2}} \big<1,\,(E_1(e_1x)e_1 )\star r^+_1(y^*) \big>\\\nonumber
&=& \tau^{-2}\tr\big(E_1(E_1(e_1x)e_1)r^+_1(y^*)\big)\quad \text{ (by \Cref{impfor e(xyz)} item $(i)$)}\\\nonumber
&=& \tau^{-1}\tr\big(E_1(e_1x)r^+_1(y^*)\big)\\\nonumber
&=& \tau^{-1}\tr\big(r^+_1( E_1(e_1x))y^*\big)\\\nonumber
&=&\tau^{-1}  \big<y,\,r^+_1(E_1(e_1x))\big>.
\end{eqnarray}
Hence we are done.
\qed
\end{prf}
\begin{lmma}\label{impforinvariantmeasure2}
For any $x,y\in B'\cap A_1$, we have
$$\big<x, y\star 1\big> = \tau^{-\frac{1}{2}}\tr(yE_1(x^*)).$$
\end{lmma}
\begin{prf}
We have the following:
\begin{eqnarray}\nonumber
\big<x, y\star 1\big> &=& \tau^{-\frac{3}{2}}\tr(x^* E_2(\mathcal{F}(1)\mathcal{F}(y)e_1e_2))\\\nonumber
&=& \tau^{-2}\tr(\mathcal{F}(y)e_1e_2x^*e_2)\\\nonumber
&=& \tau^{-\frac{5}{2}}\sum_i\tr(\lambda_i ye_2e_1\lambda^*_ie_1e_2x^*e_2)\\\nonumber
&=& \tau^{-\frac{3}{2}}\tr( ye_2x^*e_2)\\\nonumber
&=& \tau^{-\frac{1}{2}}\tr( yE_1(x^*)).
\end{eqnarray}
\qed
\end{prf}
\begin{lmma}\label{impforrightinvariantmeasure3}
For any $v\in B'\cap A$, we have $e_1\,r^+_1(v)=e_1v$.
\end{lmma}
\begin{prf}
Let \(v \in B' \cap A\). Using \Cref{reflection operator formula} and \Cref{kajiwara watatani}, we obtain
\begin{eqnarray}\nonumber
e_1\,r^+_1(v)=\tau^{-1}\sum_{i}e_1E_1(e_1 \lambda_iv)e_1\lambda^*_i
 =\sum_i e_1E_0(v\lambda_i)\lambda^*_i=e_1v.
\end{eqnarray}
\qed
\end{prf}
\begin{thm}\label{normalizedmeasures}
Let \(B \subset A\) be an inclusion of simple unital \(C^*\)-algebras with a conditional expectation of index-finite type. Then the weak quantum hypergroup \(B' \cap A_1\), as defined in \Cref{weak hypergroup structure2}, admits both left-invariant and right-invariant normalized measures.
\end{thm}
\begin{prf}
Consider the functional $\phi(x)=\tr(x).$ First we will show that $\phi$ is a left-invariant measure, i.e., $(\mathrm{id} \otimes \phi) \circ \Delta = (\varepsilon^t \otimes \phi) \circ \Delta.$ For any $y\in B'\cap A_1$, using \Cref{impforinvariantmeasure2}, we have the following:
\begin{eqnarray}\nonumber
\big<y, x_{(1)}\tr(x_{(2)})\big>&=& \big<x^*_{(1)},y^*\big>\, \big<x^*_{(2)}, 1\big>\\\nonumber
&=& \tau^{\frac{1}{2}}\big< x^*, y^* \star 1\big>\\\nonumber
&=& \tr(y^*E_1(x))\\\label{inmaintheorem31}
&=& \tr(xE_1(y^*)).
\end{eqnarray}
Moreover, using \Cref{impfor e(xyz)} item $(ii)$ together with \Cref{impforinvariantmeasure2}, we have

\begin{eqnarray}\nonumber
\big<y, \varepsilon^t(x_{(1)})\tr(x_{(2)})\big>
&=& \tau^{-1}\big<y, E_1(x_{(1)}e_1)\big>\,\big<x^*_{(2)}, 1\big>\\\nonumber
&=& \tau^{-1}E_0\big(E_1(y^*) E_1(x_{(1)}e_1)\big)\,\big<x^*_{(2)}, 1\big>\\\nonumber
&=& \tau^{-1}E_0\big( E_1(x_{(1)}e_1E_1(y^*))\big)\,\big<x^*_{(2)}, 1\big>\\\nonumber
&=&\tau^{-1}\big<x^*_{(1)}, e_1E_1(y^*)\big>\,\big<x^*_{(2)}, 1\big>\\\nonumber
&=&\tau^{-\frac{1}{2}}\big<x^*, (e_1E_1(y^*))\star 1\big>\\\nonumber
&=&\tau^{-\frac{1}{2}}\,\overline{\big<x, (E_1(y)e_1)\star 1\big>}.
\end{eqnarray}
The last equality follows by \Cref{3Proposition 3.8}. Now using \Cref{impfor e(xyz)} item $(i)$, we get that
\begin{eqnarray}\nonumber
\big<y, \varepsilon^t(x_{(1)})\tr(x_{(2)})\big>=\tau^{-1}\overline{\tr\big(E_1(E_1(y)e_1)x^*\big)}&=&\overline{\tr\big(E_1(y)x^*\big)}\\\label{inmaintheorem32}&=&\tr(xE_1(y^*)).
\end{eqnarray}
Therefore, by \Cref{inmaintheorem31} and \Cref{inmaintheorem32}, we conclude that $\phi$ is a left-invariant measure. Moreover, it is straightforward to verify that $\phi$ is normalized. Now we will show that $\phi$ is also right-invariant measure, i.e.,
$(\phi \otimes \mathrm{id}) \circ \Delta = (\phi \otimes \varepsilon^s) \circ \Delta.$ For any $y\in B'\cap A_1$, we get the following:
\begin{eqnarray}\label{impforrightinvariantmeasur1}
\big<y, \,\tr(x_{(1)})x_{(2)}\big>=\big<x^*_{(1)}, \, 1\big>\,\big<x^*_{(2)}, y^*\big>&=&\tau^{\frac{1}{2}} \big< x^*, \, 1\star y^*\big>.
\end{eqnarray}
Using \Cref{impforrightinvariantmeasure2} and the fact that $r^+_1$ is an involutive, trace preserving antihomomorphism, we deduce that
\begin{eqnarray}\nonumber
\big<y, \tr(x_{(1)})\varepsilon^s(x_{(2)})\big>&=&\tau^{-1}\big<x^*_{(1)}, \, 1\big>\,\tr(y^*r^+_1(E_1(e_1x_{(2)}))\big)\\\nonumber
&=& \tau^{-1}\big<x^*_{(1)}, \, 1\big>\,\tr\big(E_1(e_1x_{(2)})r^+_1(y^*)\big)\\\nonumber
&=& \tau^{-1}\big<x^*_{(1)}, \, 1\big>\, E_0(E_1(e_1x_{(2)})E_1(r^+_1(y^*))\big)\\\nonumber
&=& \tau^{-1}\big<x^*_{(1)}, \, 1\big>\, \big<x^*_{(2)},\,E_1(r^+_1(y^*))e_1\big>\\\nonumber
&=& \tau^{-\frac{1}{2}}\big<x^*, \, 1\star (E_1(r^+_1(y^*))e_1)\big>.
\end{eqnarray}
Now, using \cite[Proposition~3.9]{BGS} together with the fact that \(r_1^+(e_1) = e_1\), we obtain
\begin{eqnarray}\nonumber
\big<y, \tr(x_{(1)})\varepsilon^s(x_{(2)})\big>&=&\tau^{-\frac{1}{2}}\big<r^+_1(x^*),\, \big(e_1r^+_1(E_1(r^+_1(y^*)))\big) \star 1\big>\\\nonumber
&=& \tau^{-\frac{1}{2}}\big<r^+_1(x^*),\, \big(e_1E_1(r^+_1(y^*))\big) \star 1\big>\quad \text{ (by \Cref{impforrightinvariantmeasure3})}\\\nonumber
&=& \tau^{-\frac{1}{2}}\overline{\big<r^+_1(x),\, \big(E_1(r^+_1(y))e_1\big) \star 1\big>}\quad \text{ (by \Cref{3Proposition 3.8})}\\\nonumber
&=& \tau^{-1}\overline{\tr\big(E_1(E_1(r^+_1(y))e_1)r^+_1(x^*)\big)}\quad \text{ (by \Cref{impfor e(xyz)} item $(i)$)}\\\nonumber
&=& \tr\big(r^+_1(x)E_1(r^+_1(y^*)
)\big)\\\nonumber
&=&\tau^{\frac{1}{2}} \big<r^+_1(x^*) , \, r^+_1(y^*) \star 1\big>\quad \text{ (by \Cref{impforinvariantmeasure2})}\\\label{impforrightinvariantmeasur5}
&=&\tau^{\frac{1}{2}} \big<x^* , \, 1 \star y^* \big>.
\end{eqnarray}
Thus, from \Cref{impforrightinvariantmeasur1} and \Cref{impforrightinvariantmeasur5}, we obtain that $\phi$ is also a right-invariant measure. Clearly, $\phi$ is a normalized right-invariant measure.
\qed
\end{prf}
\begin{thm}\label{Haarintegralexistence}
Suppose \(B \subset A\) is an inclusion of simple unital \(C^*\)-algebras with a conditional expectation of index-finite type. Then the weak quantum hypergroup \(B' \cap A_1\), arising from \Cref{weak hypergroup structure2}, admits a Haar integral.
\end{thm}

\begin{prf}
Using \Cref{impfor e(xyz)} item $(ii)$, \Cref{impforrightinvariantmeasure2}, and \Cref{impforrightinvariantmeasure3}, it follows straightforwardly that $e_1$ is a Haar integral.
\qed
\end{prf}

\section*{Acknowledgements}
The first author acknowledges the support of the INSPIRE Faculty Grant DST/INSPIRE/04/2019/002754 and  ANRF/ECRG/2024/002328/PMS. The second author is partially supported by the
JC Bose National Fellowship awarded by ANRF, Government of India.

\noindent \\
         {\em Department of Mathematics and Statistics,
         Indian Institute of Technology Kanpur,
         Uttar Pradesh 208016, India.}\\
        { Email adress: \texttt{keshab@iitk.ac.in, bakshi209@gmail.com
}}

\noindent\\
          {\em Stat-Math Unit, Indian Statistical Institute, 203,         B.T. Road, Kolkata-700108, India.}\\
{Email address: \texttt{debashish\_goswami@yahoo.co.in}}

\noindent \\
         {\em Department of Mathematics and Statistics,
         Indian Institute of Technology Kanpur,
         Uttar Pradesh 208016, India.}\\
        { Email adress: \texttt{biplabpal32@gmail.com, bpal21@iitk.ac.in}}

\bigskip

\end{document}